\documentclass[number]{elsarticle}
\usepackage{amsfonts,amsmath,amsthm,esint,hyperref}
\usepackage[usenames,dvipsnames]{color}
\usepackage[numbers]{natbib}

\newcommand{\A}{\ve A}
\newcommand{\AEps}{\A_\eps}
\newcommand{\C}{{\mathcal{C}}}
\newcommand{\D}[1]{\tenD #1}
\DeclareMathOperator{\Djac}{D}
\newcommand{\dd}{\ve d}
\newcommand{\Deps}[1]{\tenD_\eps #1}
\newcommand{\Deriv}[2]{\frac{\mathrm{d}#1}{\mathrm{d}#2}}

\DeclareMathOperator{\dist}{dist}
\renewcommand{\div}{\mathrm{div\,}}

\newcommand{\dt}{\partial_t}
\newcommand{\dual}[2]{\langle #1,#2\rangle}
\newcommand{\eps}{\varepsilon}
\newcommand{\FF}{\ve F}
\newcommand{\ff}{\ve f}
\newcommand{\gdet}{\mathfrak g}
\newcommand{\hh}{\ve h}
\newcommand{\I}{\ten I}
\renewcommand{\L}{L}
\newcommand{\LL}{{\mathbf L}}

\newcommand{\M}{\ten M}
\newcommand{\N}{\ten N}
\newcommand{\nn}{\ve n}
\newcommand{\norm}[1]{\|#1\|}

\renewcommand{\O}{\N'}
\newcommand{\OmegaE}{{\Omega_\eps}}
\newcommand{\oOmega}{\overline\Omega}

\newcommand{\pEpsB}{{\bar p}_\eps}
\newcommand{\pOmega}{{\partial\Omega}}
\newcommand{\pOmegaE}{{\pOmega_\eps}}
\newcommand{\pphi}{\boldsymbol\phi}

\newcommand{\problem}[1]{(P(#1))}

\newcommand{\R}{{\mathbb{R}}}
\newcommand{\scal}[2]{(#1,#2)}
\DeclareMathOperator{\supp}{supp}
\newcommand{\ten}[1]{\mathbb #1}
\newcommand{\tenC}{\ten C}
\newcommand{\tenD}{\ten D}
\newcommand{\tenS}{\ten S}

\DeclareMathOperator{\tr}{tr}
\newcommand{\uEps}{\ve{u}_\eps}
\newcommand{\uu}{\ve{u}}

\newcommand{\vDot}{\tilde{\vv}}
\newcommand{\ve}[1]{{\mathbf{#1}}}
\newcommand{\vEps}{\vv_\eps}
\newcommand{\vEpsB}{{\bar\vv}_\eps}
\newcommand{\vv}{\ve v}
\newcommand{\W}{W}
\newcommand{\weakly}{\rightharpoonup}
\newcommand{\wOmega}{B}
\newcommand{\WW}{\mathbf W}
\newcommand{\ww}{\ve w}
\newcommand{\WWbc}{\WW_0}
\newcommand{\WWbcdiv}{\WW_{0,\div}}
\newcommand{\xx}{\ve x}
\newcommand{\xxi}{\boldsymbol\xi}
\newcommand{\yy}{\ve y}

\newtheorem{theorem}{Theorem}
\newtheorem{lemma}[theorem]{Lemma}

\newtheorem{assumption}{Assumption}
\newtheorem{corollary}[theorem]{Corollary}
\newtheorem*{remark}{Remark}

\begin{document}
\begin{frontmatter}

\title{Shape sensitivity analysis of time-dependent flows of~incompressible non-Newtonian fluids}

\author[iecn,pan]{J. Soko\l owski}
\ead{Jan.Sokolowski@iecn.u-nancy.fr}

\author[imas]{J. Stebel}
\ead{stebel@math.cas.cz}

\address[iecn]{Institut \'Elie Cartan, UMR 7502 Nancy-Universit\'e-CNRS-INRIA\@, Laboratoire de Ma\-th\'e\-ma\-tiques, Universit\'e Henri Poincar\'e Nancy 1, B.P.~239, 54506 Vandoeuvre L\`es Nancy Cedex, France}

\address[pan]{Systems Research Institute of the Polish Academy of Sciences, ul. Newelska 6, 01-447 Warszawa, Poland}

\address[imas]{Institute of Mathematics of the Academy of Sciences of the Czech Republic, \v Zitn\'a 25, 115~67 Praha 1, Czech Republic}

\begin{abstract}
We study the shape differentiability of a cost function for the flow of an incompressible viscous fluid of power-law type.
The fluid is confined to a bounded planar domain surrounding an obstacle.
For smooth perturbations of the shape of the obstacle we express the shape gradient of the cost function which can be subsequently used to improve the initial design.
\end{abstract}

\begin{keyword}
shape optimization, shape gradient, incompressible viscous fluid, Navier-Stokes equations


\end{keyword}
\end{frontmatter}

\section{Introduction}

Shape sensitivity analysis is a fundamental step towards the numerical solution of the shape optimization problems governed by partial differential equations. In the context of Navier-Stokes equations, the sensitivity analysis of shape functionals is performed in \cite{MR2193461, MR2737420} for the incompressible case, and in \cite{MR2683904} for the compressible case. Usually, in numerical solution of applied problems this step is made formally, see e.g., \cite{MR2791367,MR2556357} for the related results. 

The present paper is focused on the rigorous analysis and new results for a nonlinear nonstationary model in two spatial dimensions are proved. We refer the reader to \cite{MR2737420} for some results in the case of a stationary model.

We consider the time-dependent flow of an incompressible fluid in a bounded domain $\Omega:=\wOmega\setminus S$ in $\R^2$, where $\wOmega$ is a container and $S$ is an obstacle.
Motion of the fluid is described by the system of equations
\begin{align}
\notag
\dt{\vv} + \div(\vv\otimes\vv) - \div\tenS(\D{\vv}) + \nabla p + \tenC\vv & = \ff &&\mbox{ in }Q,\\
\notag
\div\vv & = 0 &&\mbox{ in }Q,\\
\label{eq:gns}\tag{$P(\Omega)$}
\vv &= 0 &&\mbox{ on }\Sigma,\\
\notag
\vv(0,\cdot) &= \vv_0 && \mbox{ in }\Omega.
\end{align}
Here $Q:=(0,T)\times\Omega$, $\Sigma:=(0,T)\times\pOmega$, where $(0,T)$ is a time interval of arbitrary length, $\vv$, $p$, $\tenC$, $\ff$ stands for the velocity, the pressure, the constant skew-symmetric Coriolis tensor and the body force, respectively.
The traceless part $\tenS$ of the Cauchy stress can depend on the symmetric part $\D\vv$ of the velocity gradient in the following way:
\begin{equation}
\label{eq:form_of_S}
\tenS(\D\vv) = \nu(|\D\vv|^2)\D\vv,
\end{equation}
where $\nu$, $|\D\vv|^2$ is the viscosity and the shear rate, respectively.
In particular, we assume that $\nu$ has a polynomial growth (see Section \ref{sec:struct-asm} below), which includes e.g. the Carreau and the power-law model.

In the model the term of Coriolis type is present. This term appears e.g. when the change of variables is performed in order to take into account the flight scenario of the obstacle in the fluid or gas.

The aim of this paper is to investigate differentiability of a shape functional depending on the solution to \eqref{eq:gns} with respect to the variations of the shape of the obstacle.
We consider a model problem with the drag functional
\begin{equation}
\label{eq:def_J}
J(\Omega):=\int_0^T\int_{\partial S}(\tenS(\D\vv)-p\I)\nn\cdot\dd,
\end{equation}
with a given constant unit vector $\dd$.
Instead of $J$ one could take other type of functional, since our method does not rely on its specific form.

Fluids whose viscosity depends on the shear rate through \eqref{eq:form_of_S} form an important class of non-Newtonian fluids (see e.g \cite{schowalter, truesdell-noll, rajagopal_mnnf} for general references).
Particular models find applications in many areas of chemistry, biology, medicine or engineering.
The first mathematical results were established already 40 years ago in \cite{ladyzhenskaya-new-equations, ladyzhenskaya-book, jl-lions-book}, for recent references see e.g. \cite{frehse_malek_steinhauer,frehse_malek_steinhauer_analysis,malek-rajagopal,diening2010existence}.
In context of optimal control the fluids with shear dependent viscosity were studied in \cite{slawig,wachsmuth-roubicek}.
Numerical results for a shape optimization problem can be found in \cite{abraham}, see also \cite{bhms,hs_papermachine}.

Our main interest is the rigorous analysis of the shape differentiability for \eqref{eq:gns} and \eqref{eq:def_J}.
We follow the general framework developed by \citet{sokolowski_zolesio} using the speed method and the notion of the material derivative.
Let us point out that due to \eqref{eq:form_of_S} the state problem is nonlinear in its nature.
In such cases the shape sensitivity analysis is usually restricted to `small data', i.e. for sufficiently small body forces, initial conditions or short times.
For the presented results no such restriction is necessary, because of uniqueness and regularity of the state variables.

\subsection{Shape derivatives}

\label{sec:shape-derivatives-formal}

We start by the description of the framework for the shape sensitivity analysis.
For this reason, we introduce a vector field $\ve T\in\C^2(\R^2,\R^2)$ vanishing in the vicinity of $\partial B$ and define the mapping
$$ \yy(\xx) = \xx + \eps\ve T(\xx). $$
For small $\eps>0$ the mapping $\xx\mapsto\yy(\xx)$ takes diffeomorphically the region $\Omega$ onto
$\OmegaE=B\setminus S_\eps$ where $S_\eps=\yy(S)$.
Denoting $Q_\eps:=(0,T)\times\OmegaE$, $\Sigma_\eps:=(0,T)\times\pOmegaE$, we consider the counterpart of problem \eqref{eq:gns} in $Q_\eps$, with the data $\ff_{|Q_\eps}$ and $\vv_{0|\OmegaE}$. The new problem will be denoted by $(P(\OmegaE))$ and its solution by $(\vEpsB,\pEpsB)$.

For the nonlinear system \eqref{eq:gns} we introduce the shape derivatives of solutions.
To this end we need the linearized system of the form:

\textit{Find the couple $(\uu,\pi)$ such that
\begin{align}
\notag
\dt{\uu} + \div\left[\uu\otimes\vv+\vv\otimes\uu-\tenS'(\D{\vv})\D{\uu}\right]
+ \nabla\pi + \tenC\uu & = \FF &&\mbox{ in }Q,\\
\notag
\div\uu & = 0 &&\mbox{ in }Q,\\
\label{eq:gnslin}\tag{$P_{\rm lin}(\Omega)$}
\uu &= \hh &&\mbox{ on }\Sigma,\\
\notag
\uu(0,\cdot) &= \uu_0 && \mbox{ in }\Omega,
\end{align}
where $ \FF $, $ \hh$ and $ \uu_0$ are given elements.
}

The shape derivative $\vv'$ and the material derivative $\dot\vv$ of solutions are formally introduced by
$$ \vv':=\lim_{\eps\to 0}\frac{\vEpsB-\vv}\eps,\quad \dot\vv:=\lim_{\eps\to 0}\frac{\vEpsB\circ\yy-\vv}\eps. $$
The standard calculus for differentiating with respect to shape yields that $\vv'$ is the solution of \eqref{eq:gnslin} with the data $\FF =\ve 0$, $ \uu_0=\ve 0$, and $\hh=-\partial \vv/\partial\nn(\ve T\cdot\nn)$.
%
Using \eqref{eq:J_volume} as the definition of $J$ we obtain the expression for the shape gradient:
\begin{multline}
\label{eq:J_volume_gradient}
dJ(\Omega;\ve T) := \lim_{\eps\to 0}\frac{J(\OmegaE)-J(\Omega)}\eps\\
= \int_\Omega\vv'(T)\cdot\xxi + \int_Q\left[ \left(\tenC\vv'\right)\cdot\xxi + \left(\tenS^\prime(\D{\vv})\D{\vv'}-\vv'\otimes\vv-\vv\otimes\vv'\right):\nabla\xxi\right]\\
- \int_0^T\int_{\partial S}(\ff\cdot\dd)\ve T\cdot\nn.
\end{multline}
In the above formula, the part containing $\vv'$ depends implicitly on the direction $\ve T$.
This is not convenient for practical use, hence we introduce the adjoint problem for further simplification of \eqref{eq:J_volume_gradient}:

\textit{Find the couple $(\ww,s)$ such that
\begin{align}
\notag
-\dt{\ww} - 2(\D\ww)\vv - \div\left[\tenS'(\D{\vv})^\top\D{\ww}\right] + \nabla s - \tenC\ww & = \ve 0 &&\mbox{ in }Q,\\
\notag
\div\ww & = 0 &&\mbox{ in }Q,\\
\label{eq:gnsadj}\tag{$P_{\rm adj}(\Omega)$}
\ww &= \dd &&\mbox{ on }\Sigma,\\
\notag
\ww(T,\cdot) &= \ve 0 && \mbox{ in }\Omega.
\end{align}
}
Consequently, the expression for $dJ$ reduces to
\begin{equation}
\label{eq:J_surface_gradient}
dJ(\Omega;\ve T) = - \int_0^T\int_{\partial S}\left[\left(\tenS'(\D\vv)^\top\D\ww-s\I\right):\dfrac{\partial\vv}{\partial\nn}\otimes\nn + \ff\cdot\dd\right]\ve T\cdot\nn.
\end{equation}

In order to prove the result given by \eqref{eq:J_volume_gradient} and \eqref{eq:J_surface_gradient} we need the material derivatives.
In particular, it is sufficient to show that the linear mapping
$$
\mathbf{T}\mapsto dJ(\Omega;\mathbf{T})
$$
is continuous in an appropriate topology, see the structure Theorem in the book \cite{sokolowski_zolesio} for details.

The paper is organized as follows:
In Section \ref{sec:prelim} we impose the structural assumptions and collect the basic facts about the existence of weak solutions to \eqref{eq:gns}.
The main results of the paper are stated in Section \ref{sec:main}.
Well-posedness of the linearized systems is studied in Section \ref{sec:well-posedness-lin}.
The rest is devoted to the proof of the main results.
In Section \ref{sec:fixed_form} the problem is formulated in the fixed domain, Section \ref{sec:stability} deals with the shape stability, Section \ref{sec:existence_mat_der} with the Lipschitz estimates and the existence of the material derivative and Section \ref{sec:shape-der} with the existence of the shape gradient of $J$.

\section{Preliminaries}
\label{sec:prelim}

We impose the structural assumptions on the data, state the known results on well-posedness of \eqref{eq:gns} and introduce the elementary notation for shape sensitivity analysis.

\subsection{Structural assumptions}
\label{sec:struct-asm}

We require that $\tenS$ has a potential $\Phi:[0,\infty)\to[0,\infty)$, i.e. $\tenS_{ij}(\ten D)=\partial\Phi(|\ten D|^2)/\partial\D_{ij}$. Further we assume that $\Phi$ is a $\C^3$ function with $\Phi(0)=0$ and that there exist constants $C_1,C_2,C_3>0$ and $r\ge 2$ such that
\begin{subequations}
\label{asm:S}
\begin{align}
\label{asm:S1}
C_1(1+|\ten A|^{r-2})|\ten B|^2 \le\tenS'(\ten A)::(\ten B\otimes\ten B) \le C_2(1+|\ten A|^{r-2})|\ten B|^2,\\
\label{asm:S2}
|\tenS''(\ten A)| \le C_3(1+|\ten A|^{r-3})
\end{align}
\end{subequations}
for any $0\neq\ten A, \ten B\in\R^{2\times 2}_{sym}$.
The above inequalities imply the following properties of $\tenS$:
\begin{lemma}
Let $\tenS$ satisfy \eqref{asm:S}.
\begin{itemize}
\item[(i)] There is a constant $C_4>0$ such that for every $\ten A\in\R^{2\times 2}_{sym}$:
\begin{equation}
\label{eq:bound_S}
|\tenS(\ten A)|\le C_4(1+|\ten A|^{r-1}).
\end{equation}
\item[(ii)] There is a constant $C_5>0$ such that
\begin{equation}
\label{eq:strong_monotony_S}
(\tenS(\ten A)-\tenS(\ten B)):(\ten A-\ten B)\ge C_5\left|\ten A-\ten B\right|^r
\end{equation}
for all $\ten A$, $\ten B\in\R^{2\times 2}_{sym}$.
\end{itemize}
\end{lemma}
The proof can be found e.g. in \cite{malek-rajagopal}.

\subsection{Weak formulation}
\label{sec:weak_form}

For the definition of the weak solution we will use the space
$$ \WWbcdiv^{1,r}(\Omega):=\{\pphi\in\WW^{1,r}_0(\Omega);~\div\pphi=0\}. $$
Let $\vv_0\in\WWbcdiv^{1,r}(\Omega)$ and $\ff\in\L^2(0,T;(\WWbcdiv^{1,2}(\Omega))^*)$.
Then a function $\vv\in\L^r(0,T;\WWbcdiv^{1,r}(\Omega))$ with $\dt\vv\in\L^2(0,T;(\WWbcdiv^{1,2}(\Omega))^*)$ is said to be a weak solution to the problem $\problem{\Omega}$ if $\vv(0)=\vv_0$ and
\begin{multline}
\label{eq:weak_form}
\int_0^T\dual{\dt{\vv}}{\pphi}_{\WWbcdiv^{1,2}(\Omega)}
+\int_Q\Big[\tenS(\D{\vv}):\D{\pphi}
-\vv\otimes\vv:\nabla\pphi
+\tenC\vv\cdot\pphi \Big]\\
=\int_Q\ff\cdot\pphi
\end{multline}
for every $\pphi\in\L^r(0,T;\WWbcdiv^{1,r}(\Omega))$.
Note that the pressure is eliminated since test functions are divergence free.

The existence and uniqueness of a weak solution to $\problem{\Omega}$ with $\tenC\equiv 0$ in two space dimensions is due to \citet{ladyzhenskaya-new-equations,ladyzhenskaya-book} and \citet{jl-lions-book}.
Regularity of weak solutions was studied e.g. by \citet{malek-necas-ruzicka, daveiga-kaplicky-ruzicka, kaplicky-regularity}.
We recall the following result:
\begin{theorem}[\citet{kaplicky-regularity}]
\label{th:regularity}
Let $\Omega\subset\R^2$ be a bounded domain with $\C^{2+\mu}$ boundary and $T>0$.
Let \eqref{asm:S} hold for some $r\in[2,4)$, $\tenC\equiv 0$ and
\begin{equation}
\label{eq:asm_f_for_regularity}
\ff\in\L^\infty(0,T;\LL^2(\Omega)), \quad \ff(0)\in\LL^2(\Omega), \quad \dt\ff\in\L^2(0,T;(\WWbcdiv^{1,2}(\Omega))^*),
\end{equation}
\begin{equation}
\label{eq:asm_v0_for_regularity}
\vv_0\in\WW^{\rho,2}(\Omega)\cap\WWbcdiv^{1,2}(\Omega), \quad \rho>2 ~(\rho=2 \mbox{ if }r=2).
\end{equation}
Then the unique weak solution $\vv$, $p$ of $\problem{\Omega}$ satisfies for $s\in(1,2)$ ($s=2$ if $r=2$):
$$ \vv\in\L^\infty(0,T;\WW^{2,s}(\Omega)), \quad \dt\vv\in\L^\infty(0,T;\LL^2(\Omega))\cap\L^2(0,T;\WWbcdiv^{1,2}(\Omega)). $$
Moreover, if there is a $\tilde q>2$ such that
\begin{equation}
\label{eq:asm_f_for_bounded_grad}
\ff\in\L^\infty(0,T;\LL^{\tilde q}(\Omega)), \quad \dt\ff\in\L^{\tilde q}(0,T;\WW^{-1,\tilde q}(\Omega)),
\end{equation}
then there exists $q>2$ and $\alpha>0$ such that for all $\epsilon\in(0,T)$ it holds:
$$ \nabla\vv,p\in\L^\infty(\epsilon,T;\W^{1,q}(\Omega))\cap\C^{0,\alpha}([\epsilon,T]\times\overline\Omega). $$
\end{theorem}

We will also need the $\L^q$ theory for the generalized Stokes system \eqref{eq:gnslin}.
Here and in what follows we will assume that the non-homogeneous boundary condition is time-independent, i.e. $\hh=\hh(\xx)$.
Theorem 4.1 in \cite{bothe-pruss} implies (see also \cite{solonnikov_lp}):

\begin{theorem}[\citet{bothe-pruss}]
\label{th:Lq_gstokes}
Let $\Omega\in\C^{2,1}$, $T>0$, $q\in(1,\infty)\setminus\{\frac32,3\}$, $\D\vv\in\C(\overline Q)$, $\FF\in\LL^q(Q)$, $\hh\in\WW^{2-1/q,q}(\pOmega)$, and $\uu_0\in\WW^{2-2/q,q}(\Omega)$.
Let $\hh$ and $\uu_0$ satisfy the compatibility conditions: $\div\uu_0=0$, $\hh\cdot\nn=\uu_0\cdot\nn$ on $\pOmega$ and for $q>\frac32$ $\hh=\uu_0$ on $\pOmega$.

Then \eqref{eq:gnslin} has a unique solution $(\uu,\pi)$ in the class
$$ \uu\in\W^{1,q}(0,T;\LL^q(\Omega))\cap\L^q(0,T;\WW^{2,q}(\Omega)), \quad \pi\in\L^q(0,T;\W^{1,q}(\Omega)). $$
\end{theorem}

The additional Coriolis term $\tenC\vv$ presents a minor technical obstacle in the existence analysis:
Since $\tenC\in\R^{2\times 2}_{skew}$, the a priori estimate remains unchanged, namely every weak solution $\vv$ satisfies the energy equality
\begin{equation}
\label{eq:energy_equality}
\frac12\norm{\vv(t)}_{2,\Omega}^2 + \int_0^t\int_\Omega\tenS(\D\vv):\D\vv = \int_0^t\int_\Omega\ff\cdot\vv + \frac12\norm{\vv_0}_{2,\Omega}^2, \mbox{ for a.a. }t\in(0,T).
\end{equation}
Considering a sequence $\{\vv^N\}$ of approximate solutions that satisfy \eqref{eq:energy_equality}, the limit passage
$$ \tenC\vv^N \weakly \tenC\vv \mbox{ in }\L^r(0,T;\WW^{1,r}(\Omega)) $$
is a straightforward consequence of strong monotonicity of $\tenS$ and the weak compactness of $\{\vv^N\}$ in $\L^r(0,T;\WW^{1,r}(\Omega))$.
In other aspects the proof of the existence and the uniqueness of a weak solution follows line by line the original one.
Concerning regularity, if we put $\tenC\vv$ to the right hand side, the assumptions of Theorem \ref{th:regularity} are still satisfied, hence the same result holds.

In order to identify the material derivative of $\vv$ and to apply Theorem \ref{th:Lq_gstokes} we will need $\nabla\vv\in\C(\overline Q)$.
Theorem~\ref{th:regularity}, however, guarantees this only for $r=2$ or locally in time.
To overcome this technical obstacle, we use the idea of \citet{wachsmuth-roubicek} and impose the following restriction on the initial condition and the body force.

\begin{assumption}
\label{asm:v0_and_f}
In what follows we suppose that $\Omega$ is a bounded domain in $\R^2$ with $\C^{2,1}$ boundary and $r\in[2,4)$,
the function $\ff$ is extended by $\ve 0$ in the complement of $Q$, and $\supp\vv_0\cap\partial S=\emptyset$.
Further we assume:

\noindent\underline{For $r=2$:}

The initial value $\vv_0$ and the body force $\ff$ satisfy \eqref{eq:asm_f_for_regularity}, \eqref{eq:asm_v0_for_regularity} and \eqref{eq:asm_f_for_bounded_grad}.

\noindent\underline{For $r>2$:}

There exists $\tau>0$ and $\tilde\ff\in\L^\infty(-\tau,T;\LL^{\tilde q}(\Omega))$ with $\dt\ff\in\L^{\tilde q}(-\tau,T;\WW^{-1,\tilde q}(\Omega))$ such that $\ff=\tilde\ff|_{Q}$;
the initial value $\vv_0$ is equal to $\vv(0)$ with $\vv$ being the solution of the problem
\begin{align*}
\dt{\vv} + \div(\vv\otimes\vv) - \div\tenS(\D{\vv}) + \nabla p + \tenC\vv & = \tilde\ff,\\
\div\vv & = 0 &&\mbox{ in }(-\tau,0)\times\Omega,\\
\vv &= 0 &&\mbox{ on }(-\tau,0)\times\pOmega,\\
\vv(-\tau,\cdot) &= \vv_{-\tau} && \mbox{ in }\Omega,
\end{align*}
where $\vv_{-\tau}$ satisfies \eqref{eq:asm_v0_for_regularity} with $\vv_{-\tau}$ in place of $\vv_0$.
\end{assumption}

\begin{corollary}
\label{cor:reg_v}
Let Assumption \ref{asm:v0_and_f} hold true.
Then there is a unique weak solution $\vv$ to \eqref{eq:gns} and the associated pressure $p$ which satisfies:
$$ \vv\in\L^\infty(0,T;\WW^{2,q}(\Omega)), $$
$$ \nabla\vv\in\C^{0,\alpha}(\overline Q), $$
$$ \dt\vv\in\L^\infty(0,T;\LL^2(\Omega))\cap\L^2(0,T;\WWbcdiv^{1,2}(\Omega)) $$
$$ p\in\L^\infty(0,T;\W^{1,q}(\Omega))\cap\C(\overline Q) $$
for some $q>2$ and $\alpha>0$.
\end{corollary}

\begin{remark}
The above result applies only to the unperturbed domain, i.e. $\eps=0$.
For the other cases we assume that the solution $(\vEpsB,\pEpsB)$ belongs just to the class of the weak solutions.
\end{remark}

Let us point out that equation \eqref{eq:def_J} which defines $J$ is not suitable for weak solutions in general, since the energy inequality does not provide enough information about the trace of $p$ and $\D\vv$.
We therefore introduce an alternative definition that requires less regularity.
Let us fix an arbitrary divergence free function $\xxi\in\C^\infty_c(B,\R^2)$ such that $\xxi=\dd$ in a vicinity of $S$.
Then, integrating \eqref{eq:def_J} by parts and using \eqref{eq:gns} yields:
\begin{equation}
\label{eq:J_volume}
J(\Omega) = \int_\Omega\left(\vv(T)-\vv_0\right)\cdot\xxi + \int_Q\left[ \left(\tenC\vv-\ff\right)\cdot\xxi + \left(\tenS(\D{\vv})-\vv\otimes\vv\right):\nabla\xxi\right].
\end{equation}
Note that this identity is finite for any $\vv\in\L^2(0,T;\WW^{1,2}(\Omega))$ and $\vv(T)\in\LL^2(\Omega)$.

\subsection{Deformation of the shape}

Arguing similarly as in Section \ref{sec:weak_form} we find that $(P(\OmegaE))$ has a unique weak solution $\vEpsB\in\L^r(0,T;\WWbcdiv^{1,r}(\OmegaE))$ with $\dt\vEpsB\in\left(\L^2(0,T;\WWbcdiv^{1,2}(\OmegaE))\right)^*$ which satisfies the energy inequality
\begin{equation}
\label{eq:energy_ineq_eps}
\frac12\norm{\vEpsB(t)}_{2,\OmegaE}^2 + \int_0^t\int_\OmegaE\tenS(\D\vEpsB):\D\vEpsB = \int_0^t\int_\OmegaE\ff\cdot\vEpsB + \frac12\norm{\vv_0}_{2,\OmegaE}^2
\end{equation}
for a.a. $t\in(0,T)$ and for $t=T$.

Let us introduce the following notation:
We will denote by $\Djac\ve T$ the Jacobian matrix whose components are $(\Djac\ve T)_{ij}=(\nabla\ve T)_{ji}=\partial_i T_j$.
Further,
$$ \N(\xx) := \gdet(\xx)\M^{-1}(\xx),\quad \M(\xx) := \I+\eps \Djac\ve T(\xx),\quad \gdet(\xx):=\det\M(\xx). $$
One can easily check that the matrix $\N$ and the determinant $\gdet$ admit the expansions:
\begin{equation}
\label{eq:expansion_gN}
\gdet = 1+\eps\div\ve T+O(\eps^2),\quad \N = \I+\eps\O+O(\eps^2), \quad \O=(\div\ve T)\I - \Djac\ve T,
\end{equation}
where the symbol $O(\eps^2)$ denotes a function whose norm in $\C^1(\oOmega)$ is bounded by $C\eps^2$.

The value of the shape functional for $\OmegaE$ is given by
$$ J(\OmegaE) := \int_\OmegaE\left(\vEpsB(T)-\vv_0\right)\cdot\xxi_\eps + \int_{Q_\eps}\left[ \left(\tenC\vEpsB-\ff\right)\cdot\xxi_\eps + \left(\tenS(\D{\vEpsB})-\vEpsB\otimes\vEpsB\right):\nabla\xxi_\eps\right], $$
where $\xxi_\eps := (\N^{-\top}\xxi)\circ\yy^{-1}$.
Using the properties of the Piola transform (see e.g. Theorem 1.7-1 in \cite{ciarlet1994mathematical}) one can check that $\div\xxi_\eps=0$.
If $\vEpsB$ and $\pEpsB$ were sufficiently smooth, it would hold that
\begin{equation}
\label{eq:formal_def_Jeps}
J(\OmegaE) = \int_0^T \int_{\partial S_\eps} (\tenS(\D\vEpsB)-\pEpsB\I)\nn\cdot\dd.
\end{equation}
Nevertheless, as opposed to \eqref{eq:gns}, we do not require any additional regularity of the solution to the perturbed problem $(P(\OmegaE))$ and hence the expression in \eqref{eq:formal_def_Jeps} need not be well defined.

We introduce the auxiliary function $\vDot$:
$$ \vDot:=\lim_{\eps\to 0}\frac{\N^\top\vEpsB\circ\yy-\vv}\eps, $$
which is related to the material derivative $\dot\vv$ by the identity
$$ \vDot = \O^\top\vv + \dot\vv. $$
For the justification of the results of the paper we will use $\vDot$ since, unlike the material derivative, it preserves the divergence free condition.

\section{Main results}
\label{sec:main}

The first step is the existence of the function $\vDot$ and hence also of the material derivative.

\begin{theorem}
\label{th:mat_der}
Let Assumption \ref{asm:v0_and_f} be satisfied.
Then the function $\vDot$ exists and is the unique weak solution of \eqref{eq:gnslin} with the data
\begin{subequations}
\label{eq:vDot}
\begin{multline}
\label{eq:def_Ap0}
\FF = \A_0' := (\O+\O^\top-\I\tr\O)\dt{\vv}
+\div(\vv\otimes\O^\top\vv) + \O\div(\vv\otimes\vv)\\
+\div\left[\tenS'(\D{\vv})\left(((\O-\I\tr\O)\nabla\vv)_{sym} - \D(\O^\top\vv)\right) + \O^\top\tenS(\D{\vv})\right]\\
-\O\div\tenS(\D{\vv})
+\left((\O-\I\tr\O)\tenC + \tenC\O^\top\right)\vv
+(\I\tr\O-\O)\ff+(\nabla\ff)\ve T,
\end{multline}
\begin{align}
\hh &= \ve 0,\\
\uu_0 &= \vDot_0:=\O^\top\vv_0 + (\nabla\vv_0)\ve T.
\end{align}
\end{subequations}
The following estimate holds:
\begin{equation}
\label{eq:est_vDot}
\norm{\vDot}_{\L^\infty(0,T;\LL^2(\Omega))\cap\L^2(0,T;\WW^{1,2}(\Omega))} \le C\norm{\A_0'}_{\L^2(0,T;\WWbcdiv^{1,2}(\Omega))^*} \le C\norm{\ve T}_{\C^2(\oOmega)}.
\end{equation}
\end{theorem}

Due to the term $(\nabla\ff)\ve T$ the right hand side $\A_0'$ is not integrable, thus one cannot apply Theorem \ref{th:Lq_gstokes}.
Instead, the well-posedness will be investigated in Section \ref{sec:well-posedness-lin}.
The next result concerns the existence of the shape gradient.

\begin{theorem}
\label{th:sh_der_J}
Under the assumptions of Theorem \ref{th:mat_der}, the shape gradient of $J$ reads
$$ dJ(\Omega,\ve T) = J_\vv(\vDot)+J_e(\ve T), $$
where the dynamical part $J_\vv$ and the geometrical part $J_e$ is given by
\begin{equation*}
J_\vv(\vDot) = \int_\Omega\left(\vDot(T)-\vDot_0\right)\cdot\xxi + \int_Q \left[ \left(\tenC\vDot\right)\cdot\xxi + \left(\tenS'(\D{\vv})\D{\vDot}-\vDot\otimes\vv-\vv\otimes\vDot\right):\nabla\xxi\right],
\end{equation*}
\begin{multline*}
J_e(\ve T) = \int_\Omega\left(\I\tr\O-\O-\O^\top\right)\left(\vv(T)-\vv_0\right)\cdot\xxi\\
+ \int_Q\Big\{ \big[\left(\I\tr\O-\O\right)\tenC\vv - \tenC\O^\top\vv
- \left(\I\tr\O-\O\right)\ff - (\nabla\ff)\ve T\big]\cdot\xxi\\
+\big[ \vv\otimes\O^\top\vv + \tenS'(\D{\vv})\left((\O\nabla\vv-\nabla(\O^\top\vv))_{sym}-(\tr\O)\D\vv\right) + \O^\top\tenS(\D\vv)\big]:\nabla\xxi\\
+ \big[\vv\otimes\vv - \tenS(\D\vv)\big]:\nabla(\O^\top\xxi) \Big\},
\end{multline*}
respectively.
In particular, as $\vDot$ depends continuously on $\ve T$, the mapping
$$ \ve T \mapsto dJ(\Omega,\ve T) $$
is a bounded linear functional on $\C^2(\R^2,\R^2)$.
\end{theorem}

Based on the previous result we can deduce that the shape gradient has the form of a distribution supported on the boundary of the obstacle.
Since this representation is unique, the formal results derived in Section \ref{sec:shape-derivatives-formal} are justified provided that the shape derivatives and adjoints exist and are sufficiently regular. This issue will be addressed in the following section.
At this point we state the final result.

\begin{corollary}
Let Assumption \ref{asm:v0_and_f} be satisfied.
Then
\begin{enumerate}
\item[(i)] the shape derivative $\vv'$ exists and is the unique weak solution to \eqref{eq:gnslin} with $\FF=\ve 0$, $\hh=-\dfrac{\partial\vv}{\partial\nn}(\ve T\cdot\nn)$, $\uu_0=\ve 0$;
\item[(ii)] the shape gradient of $J$ satisfies \eqref{eq:J_volume_gradient};
\item[(iii)] the adjoint problem \eqref{eq:gnsadj} has a unique weak solution that satisfies for arbitrary $\delta\in(0,T)$:
$$ \ww\in\L^2(0,T-\delta;\WW^{2,2}(\Omega)), \quad s\in\L^2(0,T-\delta;\W^{1,2}(\Omega)). $$
\item[(iv)] Finally, if $\ff\in\L^2(0,T;\WW^{1,2}(\Omega))$ then the shape gradient of $J$ has the representation \eqref{eq:J_surface_gradient} in the following sense:
\begin{equation}
\label{eq:dJ_adjoint_lim}
dJ(\Omega;\ve T) = - \lim_{\delta\searrow 0}\int_0^{T-\delta}\int_{\partial S}\left[\left(\tenS'(\D\vv)^\top\D\ww-s\I\right):\dfrac{\partial\vv}{\partial\nn}\otimes\nn + \ff\cdot\dd\right]\ve T\cdot\nn.
\end{equation}
\end{enumerate}
\end{corollary}

Let us note that the formula \eqref{eq:J_surface_gradient} is here replaced by \eqref{eq:dJ_adjoint_lim}.
This is due to the fact that the terminal and boundary conditions for $\ww$ do not satisfy the hypothesis of Theorem \ref{th:Lq_gstokes}.
In fact, \eqref{eq:J_surface_gradient} holds if we replace the values $\D\ww(T,\cdot)$ and $s(T,\cdot)$ by their $\L^2(\pOmega)$-limits for $t\to T$, see Lemma \ref{th:existence_adj}.

\section{Well-posedness of \eqref{eq:gnslin} and of \eqref{eq:gnsadj}}
\label{sec:well-posedness-lin}

In this section we show that the linearized system \eqref{eq:gnslin} has a unique solution under assumptions that are weaker than those of Theorem \ref{th:Lq_gstokes}.
Indeed, we have the following result.
\begin{lemma}
Let $\Omega\in\C^{0,1}$, $\vv\in\WW^{1,\infty}(Q)$, $\FF\in\L^2(0,T;\WWbcdiv^{1,2}(\Omega)^*)$, $\hh\in\L^2(0,T;\WW^{1,2}(\Omega))$, $\dt\hh\in\L^2(0,T;\WWbcdiv^{1,2}(\Omega)^*)$ and $\uu_0-\hh(0,\cdot)\in\LL^2_{0,\div}(\Omega)$.
Then \eqref{eq:gnslin} admits a unique weak solution which satisfies the estimate:
\begin{multline*}
\norm{\uu}_{\L^\infty(0,T;\LL^2(\Omega))\cap\L^2(0,T;\WWbcdiv^{1,2}(\Omega))} \le C(\norm{\uu_0}_{2,\Omega}+\norm{\FF}_{\L^2(0,T;\WWbcdiv^{1,2}(\Omega)^*)}\\
+\norm{\hh}_{\L^2(0,T;\WW^{1,2}(\Omega))}+\norm{\dt\hh}_{\L^2(0,T;\WWbcdiv^{1,2}(\Omega)^*)}).
\end{multline*}
\end{lemma}
\begin{proof}
Using the standard Galerkin method, one checks that the approximate solutions $\uu^N$ satisfy
$$ \norm{\uu^N}_{\L^\infty(0,T;\LL^2(\Omega))\cap\L^2(0,T;\WW^{1,2}(\Omega))} \le C, $$
where the constant on the r.h.s. depends on $\norm{\vv}_{1,\infty,Q}$ and the respective norms of $\uu_0$, $\hh$, $\FF$ and $|\tenC|$.
Hence there is a weak limit $\uu$ of $\{\uu^N\}$ which is a weak solution to \eqref{eq:gnslin}.
The uniqueness can be proved testing by the difference $\uu_1-\uu_2$ of two solutions, from which one gets for a.a. $t\in(0,T)$:
$$ \norm{\uu_1(t)-\uu_2(t)}_{2,\Omega}^2 + \norm{\D(\uu_1-\uu_2)}_{2,Q}^2 \le C\norm{\vv}_{\infty,Q}\norm{\uu_1-\uu_2}_{2,Q}\norm{\nabla(\uu_1-\uu_2)}_{2,Q}. $$
Gronwall's identity then directly implies $\uu_1=\uu_2$.
\end{proof}
A direct consequence of the lemma is the existence and uniqueness of the shape derivative $\vv'$.

Applying the same technique as in the previous lemma and additionally Theorem \ref{th:Lq_gstokes}, one obtains the result for the adjoint problem:
\begin{lemma}
\label{th:existence_adj}
Let $\Omega\in\C^{0,1}$, $\xxi\in\C^\infty(\overline B)$.
Then there is a unique weak solution $(\ww,s)$ to \eqref{eq:gnsadj}.
If in addition $\Omega\in\C^{2,1}$, then for any $\delta\in(0,T)$:
$$ \ww\in\L^2(0,T-\delta;\WW^{2,2}(\Omega)),\quad\dt\ww\in\LL^2(0,T-\delta;\LL^2(\Omega)),\quad s\in\L^2(0,T-\delta;\W^{1,2}(\Omega)). $$
In particular, $\lim_{\delta\searrow 0}\D\ww(T-\delta,\cdot)$ and $\lim_{\delta\searrow 0}s(T-\delta,\cdot)$ in $\L^2(\pOmega)$ exist and are finite.
\end{lemma}
\begin{proof}
Using the change of variables $t\mapsto T-t$ we transform \eqref{eq:gnsadj} to \eqref{eq:gnslin} and obtain uniqueness of the adjoint state $\ww$.
It is however not possible to apply Theorem \ref{th:Lq_gstokes} since the terminal condition for $\ww$ is nonzero on $\partial S$.
Nevertheless, considering the adjoint problem on the time interval $(0,T-\delta)$, we find that the terminal condition $\ww(T-\delta,\cdot)$ is compatible and hence
$$ \norm{\D\ww}_{2,(0,T-\delta)\times\pOmega} + \norm{s}_{2,(0,T-\delta)\times\pOmega} \le C|\dd|. $$
Existence of the limits of $\D\ww$ and $s$ follows from the continuous dependence on the data.
\end{proof}

The rest of the paper is devoted to the proof of Theorems \ref{th:mat_der} and \ref{th:sh_der_J}.

\section{Formulation in the fixed domain}
\label{sec:fixed_form}

In this section we transform the problem $\problem{\OmegaE}$ to the fixed domain $\Omega$.
Let us introduce the following notation:
$$ \vEps(t,\xx):=\N^\top(\xx)\vEpsB(t,\yy(\xx)),\quad \vv_{0\eps}(\xx) := \N^\top(\xx)\vv_0(\yy(\xx)), ~(t,\xx)\in Q.$$
Note that the definition of $\vEps$ implies that $\div\vEps=0$.
The new function $\vEps\in\L^r(0,T;\WWbcdiv^{1,r}(\Omega))$ satisfies $\gdet\N^{-1}\N^{-\top}\dt\vEps\in\left(\L^r(0,T;\WWbcdiv^{1,r}(\Omega))\right)^*$, $\vEps(0,\cdot)=\vv_{0\eps}$ and the equality
\begin{multline}
\label{eq:vEps}
\int_0^T\dual{\gdet\N^{-1}\N^{-\top}\dt\vEps}{\pphi}_{\WWbcdiv^{1,r}(\Omega)} + \int_Q\Big[\gdet\tenS(\Deps\vEps):\Deps\pphi - \vEps\otimes\vEps:\nabla\pphi + \tenC\vEps\cdot\pphi\Big]\\ = \int_Q\ff\cdot\pphi + \int_0^T\dual{\AEps^1}{\pphi}_{\WWbcdiv^{1,2}(\Omega)} ~\mbox{ for all }\pphi\in\L^r(0,T;\WWbcdiv^{1,r}(\Omega)),
\end{multline}
where the term $\AEps^1$ on the right hand side is defined for $\pphi\in\L^2(0,T;\WWbcdiv^{1,2}(\Omega))$ by
\begin{multline}
\label{eq:def_AEps1}
\int_0^T\dual{\AEps^1}{\pphi}_{\WWbcdiv^{1,2}(\Omega)} = \int_Q\Big[\vEps\otimes\N^{-\top}\vEps:\nabla(\N^{-\top}\pphi) - \vEps\otimes\vEps:\nabla\pphi
\\
+ (\tenC-\gdet\N^{-1}\tenC\N^{-\top})\vEps\cdot\pphi + (\gdet\N^{-1}\ff\circ\yy-\ff)\cdot\pphi\Big].
\end{multline}
Here
$$ \Deps\vEps := \gdet^{-1}(\N\nabla(\N^{-\top}\vEps))_{sym}. $$
The left hand side of \eqref{eq:vEps} contains the perturbed term $\gdet\N^{-1}\N^{-\top}\dt\vEps$ due to the lack of a uniform estimate for $\dt\vEps$.
Similarly, the perturbed elliptic term $\gdet\tenS(\Deps\vEps):\Deps\pphi$ is present because of insufficient Lipschitz estimates, see Section~\ref{sec:existence_mat_der} for more details.

Applying change of coordinates we get:
\begin{multline}
\label{eq:Jeps_volume}
J(\OmegaE) = \int_\Omega \gdet\N^{-1}\N^{-\top}\left(\vEps(T)-\vv_{0\eps}\right)\cdot\xxi + \int_Q \Big[ \gdet\left(\N^{-1}\tenC\N^{-\top}\vEps - \N^{-1}\ff\circ\yy\right)\cdot\xxi\\
+ \left(\N^\top\tenS(\Deps\vEps) - \vEps\otimes(\N^{-\top}\vEps)\right):\nabla(\N^{-\top}\xxi) \Big].
\end{multline}

Now after all quantities and equations have been transformed to the fixed domain $\Omega$, we can analyze the limit $\eps\to 0$.

\section{Shape stability of weak solutions}
\label{sec:stability}

In this section we prove that $\vEps$ converges to $\vv$ in certain sense.
The result will be applied in the forthcoming sections.

\subsection{Uniform estimates}
Since the strong monotonicity of $\tenS$ and the Korn inequality hold uniformly for $\eps\to 0$, the energy inequality \eqref{eq:energy_ineq_eps} implies that $\norm{\vEpsB}_{\L^\infty(0,T;\LL^2(\OmegaE))}$, $\norm{\vEpsB}_{\L^r(0,T;\WWbcdiv^{1,r}(\OmegaE))}$ and $\norm{\vEpsB(T)}_{2,\OmegaE}$ is bounded uniformly with respect to $\eps$.
Using the expansions \eqref{eq:expansion_gN} we realize that
$$ \norm{\vEpsB(T)}_{2,\OmegaE} = \norm{\vEps(T)}_{2,\Omega} + o(1), $$
where $o(1)\to 0$ as $\eps\to 0$.
The same holds for the other norms, hence we derive the uniform estimates of $\vEps$:
\begin{equation}
\label{eq:bounds_vEps}
\{\vEps\}_{\eps>0} \mbox{ is bounded in }\L^\infty(0,T;\LL^2(\Omega)) \mbox{ and in } \L^r(0,T;\WWbcdiv^{1,r}(\Omega)),
\end{equation}
$$ \{\vEps(T)\}_{\eps>0} \mbox{ is bounded in }\LL^2(\Omega). $$
The Lebesgue-Sobolev interpolation inequality
\begin{equation}
\label{eq:lebesgue_sobolev_interpol}
\norm{\vEps}_{2r}^2 \le C\norm{\nabla\vEps}_r\norm{\vEps}_2
\end{equation}
together with \eqref{eq:bounds_vEps} yields that
\begin{equation}
\label{eq:Lq_bound_vEps}
\{\vEps\}_{\eps>0} \mbox{ is bounded in }\LL^{2r}(Q).
\end{equation}
Using this information and \eqref{eq:bound_S} we can estimate all but the first terms in \eqref{eq:vEps} and thus
$$ \{\gdet\N^{-1}\N^{-\top}\dt\vEps\}_{\eps>0} \mbox{ is bounded in }\left(\L^r(0,T;\WWbcdiv^{1,r}(\Omega))\right)^*. $$

\subsection{Convergence}
The uniform bounds, the interpolation inequality and the Aubin-Lions argument give rise to the following convergence:
\begin{align}
\notag
\vEps &\weakly \bar\vv &&\mbox{weakly-* in }\L^\infty(0,T;\LL^2(\Omega)),\\
\notag
                  &&&\mbox{weakly in }\L^r(0,T;\WWbcdiv^{1,r}(\Omega)),\\
\label{eq:weak_conv_vEps_Lq}
                  &&&\mbox{weakly in }\LL^{2r}(Q),\\
\label{eq:strong_conv_vEps}
\vEps &\to \bar\vv && \mbox{strongly in }\LL^z(Q),~1\le z<2r,\\
\notag
\vEps(T) &\weakly \bar\vv(T) && \mbox{weakly in }\LL^2(\Omega),\\
\notag
\gdet\N^{-1}\N^{-\top}\dt\vEps &\weakly \dt\bar\vv &&\mbox{weakly in }\left(\L^r(0,T;\WWbcdiv^{1,r}(\Omega))\right)^*,\\
\notag
\N^\top\tenS(\Deps\vEps) &\weakly \overline{\tenS(\D\vv)} &&\mbox{weakly in }\L^{r'}(Q,\R^{2\times 2}),\\
\notag
\AEps^1 &\weakly \ve 0 &&\mbox{weakly in }\left(\L^r(0,T;\WWbcdiv^{1,r}(\Omega))\right)^*,
\end{align}
where $\bar\vv$, $\overline{\tenS(\D\vv)}$ satisfy the integral identity
$$ \int_0^T\dual{\dt\bar\vv}{\pphi}_{\WWbcdiv^{1,r}(\Omega)} + \int_Q\left[\overline{\tenS(\D\vv)}:\D\pphi - \bar\vv\otimes\bar\vv:\nabla\pphi + \tenC\bar\vv\cdot\pphi \right] = \int_Q\ff\cdot\pphi $$
with the test functions $\pphi\in\L^r(0,T;\WWbcdiv^{1,r}(\Omega))$.
Clearly $\bar\vv$ satisfies the initial condition $\bar\vv(0,\cdot)=\vv_0$.
To identify the weak limit $\overline{\tenS(\D\vv)}$ we show that
\begin{equation}
\label{eq:strong_conv_Dveps}
\D{\vEps}\to\D\vv \mbox{ strongly in }\L^r(Q).
\end{equation}
From \eqref{eq:strong_monotony_S} we get:
\begin{align*}
C_5\norm{\D(\vEps-\bar\vv)}_{r,Q}^r
&\le \int_Q\left(\tenS(\D\vEps)-\tenS(\D\vv)\right):\D(\vEps-\bar\vv)\\
&= \int_Q\left(\tenS(\D\vEps)-\tenS(\Deps\vEps)\right):\D(\vEps-\bar\vv)\\
&+ \int_Q\tenS(\Deps\vEps):\left(\D(\vEps-\bar\vv)-(\N\nabla(\N^{-\top}(\vEps-\bar\vv)))_{sym}\right)\\
&+ \int_Q\tenS(\Deps\vEps):(\N\nabla(\N^{-\top}(\vEps-\bar\vv)))_{sym}\\
&- \int_Q\tenS(\D\bar\vv):\D(\vEps-\bar\vv)\\
&= I_1 + I_2 + I_3 + I_4.
\end{align*}
We immediately see that $I_4\to 0$ as $\eps\to 0$.
Further $\Deps\vEps=\D\vEps+\eps\ten A_\eps'$, where $\norm{\ten A_\eps'}_{r,Q}\le C$ uniformly w.r.t. $\eps\ge 0$, as follows from the uniform estimates for $\vEps$.
Consequently it holds:
\begin{multline*}
I_1 = \int_Q\int_0^1\Deriv{}{s}\tenS(\Deps\vEps+s(\D\vEps-\Deps\vEps))ds:\D(\vEps-\bar\vv)\\
= \int_Q\int_0^1\tenS'(\Deps\vEps+s(\D\vEps-\Deps\vEps))ds (\D\vEps-\Deps\vEps):\D(\vEps-\bar\vv)\\
\le \norm{\int_0^1\tenS'(\Deps\vEps+s(\D\vEps-\Deps\vEps))ds}_{\frac{r}{r-2},Q} \norm{\eps\ten A_\eps'}_{r,Q} \norm{\D(\vEps-\bar\vv)}_{r,Q}\\
\le C\eps,
\end{multline*}
where we have also used \eqref{asm:S1}.
In the same spirit we obtain:
$$ I_2 \le C\eps. $$
The last term $I_3$ can be expressed using \eqref{eq:vEps} as
\begin{equation*}
I_3 = -\int_0^T\dual{\gdet\N^{-1}\N^{-\top}\dt{\vEps}}{\vEps-\bar\vv}_{\WWbcdiv^{1,r}(\Omega)}
+o(1),
\end{equation*}
where $o(1)\to 0$ as $\eps\to 0$, as follows from the available convergence and energy estimates.
Integrating by parts with respect to time we get:
\begin{multline*}
I_3 = \norm{\sqrt\gdet\N^{-\top}(\vv_{0\eps}-\vv_0)}_2^2
-\norm{\sqrt\gdet\N^{-\top}(\vEps(T)-\vv(T))}_2^2\\
-\int_0^T\dual{\gdet\N^{-1}\N^{-\top}\dt\vv}{\vEps-\vv}_{\WWbcdiv^{1,r}(\Omega)} + o(1) = I_{31} + I_{32} + I_{33} + o(1),
\end{multline*}
where $I_{31}\to 0$, $I_{32}\le 0$ and $I_{33}$ vanishes by the weak convergence of $\vEps$ in $\L^r(0,T;\WWbc^{1,r}(\Omega))$.
We have proved \eqref{eq:strong_conv_Dveps}.

By means of the Vitali theorem we conclude that $\overline{\tenS(\D\vv)}=\tenS(\D\bar\vv)$ and consequently $\bar\vv=\vv$ is the weak solution to $\problem{\Omega}$.
As this solution is unique, the whole sequence $\{\vEps\}$ converges to $\vv$.

\section{Existence of material derivative}
\label{sec:existence_mat_der}

Our next task is to identify $\vDot$ as the limit of the sequence $\{\uEps\}$, where
$$ \uEps:=\frac{\vEps-\vv}\eps. $$
First we write down the system for the differences $\uEps$.
Subtracting \eqref{eq:vEps} and \eqref{eq:weak_form} we find that $\uEps\in\L^r(0,T;\WWbcdiv^{1,r}(\Omega))$ satisfies $\gdet\N^{-1}\N^{-\top}\dt\uEps\in\left(\L^r(0,T;\WWbcdiv^{1,r}(\Omega))\right)^*$, $\uEps(0,\cdot)=\frac{\vv_{0\eps}-\vv_0}\eps$ and the equality
\begin{multline}
\label{eq:uEps}
\int_0^T\dual{\gdet\N^{-1}\N^{-\top}\dt\uEps}{\pphi}_{\WWbcdiv^{1,r}(\Omega)} + \int_Q\Big[\frac1\eps\gdet(\tenS(\Deps\vEps)-\tenS(\Deps\vv)):\Deps\pphi\\
+ \tenC\uEps\cdot\pphi - \left(\vEps\otimes\uEps+\uEps\otimes\vv\right):\nabla\pphi \Big]
= \frac1\eps\int_0^T\dual{\AEps}{\pphi}_{\WWbcdiv^{1,2}(\Omega)}
\end{multline}
for all $\pphi\in\L^r(0,T;\WWbcdiv^{1,r}(\Omega))$.
The term $\AEps\in\left(\L^2(0,T;\WWbcdiv^{1,2}(\Omega))\right)^*$ on the right hand side is defined as follows:
$$ \AEps := \AEps^1 + \AEps^2 + \AEps^3, $$
$$ \AEps^1 \mbox{ is given by \eqref{eq:def_AEps1}}, $$
$$ \int_0^T\dual{\AEps^2}{\pphi}_{\WWbcdiv^{1,2}(\Omega)} := \int_Q(\I-\gdet\N^{-1}\N^{-\top})\dt\vv\cdot\pphi, $$
\begin{equation*}
\int_0^T\dual{\AEps^3}{\pphi}_{\WWbcdiv^{1,2}(\Omega)} := \int_Q\Big[\N^\top\tenS(\Deps\vv):\nabla(\N^{-\top}\pphi)-\tenS(\D\vv):\D{\pphi}\Big].
\end{equation*}

Next we want to derive uniform bounds for $\{\uEps\}$.
Therefore we refine the estimates of the previous section.
In contrast to Section \ref{sec:stability}, these estimates do not follow the structure of the problem, namely the elliptic term $\frac1\eps\int_Q\gdet(\tenS(\Deps\vEps)-\tenS(\Deps\vv)):\Deps(\cdot)$ is bounded in $(\L^{\frac{2r}{4-r}}(0,T;\WWbcdiv^{1,\frac{2r}{4-r}}(\Omega)))^*$ and the right hand side $\frac1\eps\AEps$ in $(\L^2(0,T;\WWbcdiv^{1,2}(\Omega)))^*$ only.

\subsection{Lipschitz estimates for \texorpdfstring{$\AEps$}{Aeps}}
\label{sec:lip_est_Aeps}
In this subsection we will investigate the Lipschitz continuity of the map $\eps\mapsto\AEps$.
In particular, we are going to show that
\begin{equation}
\label{eq:conv_fracAepsEps}
\frac\AEps\eps \weakly \A_0' \mbox{ weakly in }\left(\L^2(0,T;\WWbcdiv^{1,2}(\Omega))\right)^*,
\end{equation}
where $\A_0'$ is defined in \eqref{eq:def_Ap0}.

Indeed, due to \eqref{eq:expansion_gN} and Corollary \ref{cor:reg_v} it holds:
$$ \frac1\eps\AEps^2 \to (\O+\O^\top-\I\tr\O)\dt\vv \mbox{ strongly in }\LL^2(Q). $$
In order to pass to the limit in the terms $\frac1\eps\AEps^1$ and $\frac1\eps\AEps^3$ we rewrite them in a convenient form:
\begin{multline*}
\frac1\eps\int_0^T\dual{\AEps^1}{\pphi}_{\WWbcdiv^{1,2}(\Omega)} =
\int_Q\vEps\otimes\vEps:\nabla\left(\frac{\N^{-\top}-\I}\eps\pphi\right)
+\int_Q\vEps\otimes\frac{\N^{-\top}-\I}\eps\vEps:\nabla\left(\N^{-\top}\pphi\right)\\
+\int_Q\frac{\I-\gdet\N^{-1}}\eps\tenC\vEps\cdot\pphi
+\int_Q\gdet\N^{-1}\tenC\frac{\I-\N^{-\top}}\eps\vEps\cdot\pphi\\
+\int_Q\frac{\gdet\N^{-1}-\I}\eps(\ff\circ\yy)\cdot\pphi
+\int_Q\frac{\ff\circ\yy-\ff}\eps\cdot\pphi
=\sum_{j=1}^6 I_j,
\end{multline*}
\begin{multline*}
\frac1\eps\int_Q\dual{\AEps^3}{\pphi}_{\WWbcdiv^{1,2}(\Omega)} = \int_Q\tenS(\D\vv):\nabla\left(\frac{\N^{-\top}-\I}\eps\pphi\right)
+\int_Q\frac{\N^\top-\I}\eps\tenS(\D\vv):\nabla\left(\N^{-\top}\pphi\right)\\
+\int_Q\N^\top\frac{\tenS(\Deps\vv)-\tenS(\D\vv)}\eps:\nabla\left(\N^{-\top}\pphi\right) = \sum_{j=7}^9 I_j.
\end{multline*}
Employing \eqref{eq:weak_conv_vEps_Lq} we obtain for $\eps\to 0$:
\begin{align*}
I_1 &\to \int_Q\O\div(\vv\otimes\vv)\cdot\pphi,\\
I_2 &\to \int_Q\div\left(\vv\otimes\O^\top\vv\right)\cdot\pphi.
\end{align*}
Similarly, from \eqref{eq:strong_conv_vEps} it follows that
\begin{align*}
I_3 &\to \int_Q(\O-\I\tr\O)\tenC\vv\cdot\pphi,\\
I_4 &\to \int_Q\tenC\O^\top\vv\cdot\pphi,\\
I_5 &\to \int_Q(\I\tr\O-\O)\ff\cdot\pphi.
\end{align*}
Analogously,
\begin{align*}
I_7 &\to -\int_Q\tenS(\D\vv):\nabla\left(\O^\top\pphi\right),\\
I_8 &\to \int_Q\O^\top\tenS(\D\vv):\D\pphi.
\end{align*}
Rewriting $I_6$ carefully yields:
\begin{multline*}
I_6 = \frac1\eps\int_Q\left[(\ff\cdot\pphi)\circ\yy + \ff\circ\yy\cdot(\pphi-\pphi\circ\yy) - \ff\cdot\pphi\right]\\
= \int_{Q\cap Q_\eps}\frac{\gdet^{-1}-1}\eps\ff\cdot\pphi - \frac1\eps\int_{Q\setminus Q_\eps}\ff\cdot\pphi + \int_Q\ff\circ\yy\cdot\frac{\pphi-\pphi\circ\yy}\eps\\
\to -\int_Q\left[(\ff\cdot\pphi)\div\ve T + \nabla\pphi:\ff\otimes\ve T\right] =: \int_0^T\dual{(\nabla\ff)\ve T}{\pphi}_{\WWbc^{1,2}(\Omega)},
\end{multline*}
where we have used the Hardy inequality and the following argument:
$$ \frac1\eps\left|\int_{Q\setminus Q_\eps}\ff\cdot\pphi\right| \le C\int_{Q\setminus Q_\eps}\frac{|\ff\cdot\pphi|}{\dist_\pOmega} \le C\norm{\ff}_{2,Q\setminus Q_\eps}\norm{\nabla\pphi}_{2,Q} \to 0. $$

Finally, taking into account the continuity of $\tenS'$ and $\D\vv$ (as follows from Corollary \ref{cor:reg_v}), we have:
\begin{multline*}
I_9 = \N^{-1}\div\left(\N^\top\int_0^1\tenS'(\Deps\vv+s(\D\vv-\Deps\vv))ds\frac{\D\vv-\Deps\vv}\eps\right)\\
\to \div\left(\tenS'(\D\vv)\left(\tr\O\D\vv-(\O\nabla\vv)_{sym}+\D(\O^\top\vv)\right)\right)
\end{multline*}
strongly in $\left(\L^2(0,T;\WWbcdiv^{1,2}(\Omega))\right)^*$.
Summing all terms up we arrive at \eqref{eq:conv_fracAepsEps}.

\subsection{Estimates of \texorpdfstring{$\uEps$}{ueps}}
Testing \eqref{eq:uEps} by $1_{[0,t)}\uEps$ and integrating by parts the time derivative and the convective terms we obtain:
\begin{multline*}
\frac12\norm{\sqrt\gdet\N^{-\top}\uEps(t)}_2^2 + \int_0^t\int_\Omega\uEps\otimes\uEps:\nabla\vEps + \frac1\eps\int_0^t\int_\Omega\gdet(\tenS(\Deps\vEps)-\tenS(\Deps\vv)):\Deps\uEps\\
= \frac12\norm{\sqrt\gdet\N^{-\top}\uEps(0)}_2^2 + \frac1\eps\int_0^t\dual{\AEps}{\uEps}_{\WWbcdiv^{1,2}(\Omega)}.
\end{multline*}
Next we apply \eqref{eq:strong_monotony_S}, \eqref{eq:expansion_gN}, \eqref{eq:lebesgue_sobolev_interpol}, H\"older's and Young's inequality to show that
\begin{multline*}
\norm{\uEps(t)}_2^2 + \int_0^t\norm{\D\uEps}_2^2
\le C\Big(\int_0^T\norm{\nabla\vEps}_2\norm{\nabla\uEps}_2\norm{\uEps}_2^2 + \norm{\vv_0}_2^2\\
+ \frac1\eps\int_0^T\norm{\AEps}_{\left(\L^2(0,T;\WWbcdiv^{1,2}(\Omega))\right)^*}^2\Big)
\end{multline*}
for a.a. $t\in(0,T)$ and for $t=T$.
Gronwall's and Korn's inequality and \eqref{eq:conv_fracAepsEps} then directly implies that
$$ \{\uEps\}_{\eps>0} \mbox{ is bounded in }\L^\infty(0,T;\LL^2(\Omega)) \mbox{ and in }\L^2(0,T;\WW^{1,2}(\Omega)), $$
$$ \{\uEps(T)\}_{\eps>0} \mbox{ is bounded in }\LL^2(\Omega). $$
Since $\D\vEps$ is bounded only in $\L^r(Q)$, we have for $\pphi\in\L^{\frac{2r}{4-r}}(0,T;\WW^{1,\frac{2r}{4-r}}(\Omega))$:
\begin{multline*}
\frac1\eps\int_{Q}\gdet(\tenS(\Deps\vEps)-\tenS(\Deps\vv)):\Deps\pphi = \int_{Q}\gdet\int_0^1\tenS'(\Deps\vv+s\Deps(\vEps-\vv))ds\Deps\uEps:\Deps\pphi\\
\le C(\norm{\D\vEps}_r+\norm{\D\vv}_r)^{r-2}\norm{\D\uEps}_2\norm{\D\pphi}_{\frac{2r}{4-r}}
\end{multline*}
and thus
\begin{equation}
\label{eq:bound_dif_SDv}
\left\{\frac{\tenS(\Deps\vEps)-\tenS(\Deps\vv)}\eps\right\}_{\eps>0} \mbox{ is bounded in }\L^{\frac{2r}{3r-4}}(Q).
\end{equation}
Consequently, due to \eqref{eq:uEps},
\begin{equation}
\label{eq:bound_dtuEps}
\left\{\gdet\N^{-1}\N^{-\top}\dt\uEps\right\}_{\eps>0} \mbox{ is bounded in }\left(\L^{\frac{2r}{4-r}}(0,T;\WWbcdiv^{1,\frac{2r}{4-r}}(\Omega))\right)^*.
\end{equation}
%

\subsection{Convergence of \texorpdfstring{$\uEps$}{ueps}}
We are ready to prove the existence of material derivatives.
The estimates derived in the previous section imply that
\begin{subequations}
\label{eq:conv_uEps}
\begin{align}
\uEps &\weakly \bar\uu &&\mbox{ weakly-* in }\L^\infty(0,T;\LL^2(\Omega)),\\
      &              &&\mbox{ weakly in }\L^2(0,T;\WWbcdiv^{1,2}(\Omega)),\\
      &              &&\mbox{ weakly in }\LL^4(Q),\\
\uEps(T)&\weakly \bar\uu(T) &&\mbox{ weakly in }\LL^2(\Omega),\\
\uEps &\to     \bar\uu &&\mbox{ strongly in }\LL^{z}(Q),~ 1\le z<4.
\end{align}
Since $\tenS\in\C^2$, the Nemytski\u\i\ mapping associated to $\tenS'$ is continuous from $\L^r$ to $\L^{\frac{r}{r-2}}$, hence from \eqref{eq:strong_conv_Dveps} we deduce that
\begin{equation}
\frac1\eps\int_0^T\scal{\gdet(\tenS(\Deps\vEps)-\tenS(\Deps\vv))}{\Deps\pphi} \to \int_0^T\scal{\tenS'(\D\vv)\D\bar\uu}{\D\pphi}
\end{equation}
for all $\pphi\in\L^{\frac{2r}{4-r}}(0,T;\WW^{1,\frac{2r}{4-r}}(\Omega))$.
Consequently, \eqref{eq:bound_dtuEps} yields:
\begin{equation}
\gdet\N^{-1}\N^{-\top}\dt\uEps \weakly \dt\bar\uu
\end{equation}
\end{subequations}
weakly in $\left(\L^{\frac{2r}{4-r}}(0,T;\WWbcdiv^{1,\frac{2r}{4-r}}(\Omega))\right)^*$.
The initial conditions satisfy:
$$ \uEps(0) \weakly \O^\top\vv_0 + (\nabla\vv_0)\ve T=\bar\uu(0) \mbox{ weakly in }\LL^2(\Omega), $$
hence $\bar\uu=\vDot$ is the unique weak solution to \eqref{eq:vDot}.
Theorem \ref{th:mat_der} has been proved.

\section{Shape gradient of $J$}
\label{sec:shape-der}

Let us decompose the fraction
$$ \frac{J(\Omega_\eps)-J(\Omega)}\eps $$
into the sum $J_1^\eps+J_2^\eps$ where, in view of \eqref{eq:J_volume} and \eqref{eq:Jeps_volume},
\begin{multline*}
J_1^\eps:=\int_\Omega\gdet\N^{-1}\N^{-\top}\left(\uEps(T)-\uEps(0)\right)\cdot\xxi + \int_Q\bigg[\gdet\N^{-1}\tenC\N^{-\top}\uEps\cdot\xxi\\
+\left(\N^\top\frac{\tenS(\Deps\vEps)-\tenS(\D\vv)}\eps
-\uEps\otimes\vEps
-\vv\otimes\uEps\right):\nabla(\N^{-\top}\xxi)\bigg]
\end{multline*}
and
\begin{multline*}
J_2^\eps:=\int_\Omega \frac{\gdet\N^{-1}\N^{-\top}-\I}\eps\left(\vv(T)-\vv_0\right)\cdot\xxi + \int_Q\bigg[\Big(\gdet\N^{-1}\tenC\frac{\N^{-\top}-\I}\eps\vv
+\frac{\gdet\N^{-1}-\I}\eps\tenC\vv\\
-\frac{\gdet\N^{-1}-\I}\eps\bar\ff_\eps\circ\yy
-\frac{\ff\circ\yy-\ff}\eps\Big)\cdot\xxi\\
+\left(\frac{\N^\top-\I}\eps\tenS(\D\vv)
-\vEps\otimes\frac{\N^{-\top}-\I}\eps\vEps\right):\nabla(\N^{-\top}\xxi)\\
+\left(\tenS(\D\vv)-\vv\otimes\vv\right):\nabla\left(\frac{\N^{-\top}-\I}\eps\xxi\right)\bigg].
\end{multline*}
From \eqref{eq:bound_dif_SDv} and \eqref{eq:conv_uEps} it follows that
$$ J_1^\eps\to J_\vv(\vDot). $$
Using similar arguments as in Section \ref{sec:lip_est_Aeps} we can show that
$$ J_2^\eps\to J_e(\ve T). $$
The continuity of the map $\ve T\mapsto dJ(\Omega;\ve T)$ follows from the estimate \eqref{eq:est_vDot}.
Theorem \ref{th:sh_der_J} is proved.

\section{Conclusion}

The paper was devoted to the proof of shape differentiability for the problem of time-dependent planar flow of an incompressible fluid with shear rate dependent viscosity.
We have considered the so-called subcritical case when the weak solution is known to be unique and have bounded gradient.
These strong properties were crucial in the proof of the existence of the material derivative, in particular for the well-posedness of the linearized system satisfied by $\vDot$.
For this reason it seems impossible to extend the results to 3 spatial dimensions and to $r<2$, because no available theory can guarantee bounded gradient of the solution in such cases.


The considered model itself has a limited practical applicability, mainly due to the homogeneous boundary condition.
We believe that a generalization to other types of boundary conditions and shape functionals is possible to some extent, however to keep ideas clear the slightly artificial setting was chosen.
Extension to the problems with non-trivial boundary conditions as well as application to numerical methods will be considered in the further research.

\section*{Acknowledgement}

The work of J. Stebel was supported by the Czech Science Foundation (GA\v CR) grant No. 201/09/0917 and by the ESF grant Optimization with PDE Constraints.

\bibliographystyle{model1b-num-names}
\bibliography{ref}

\end{document}